\documentclass[11pt]{article}

\usepackage{amsmath,amssymb,amsthm}
\usepackage[a4paper,margin=1in]{geometry}
\usepackage{booktabs}
\usepackage{enumitem}
\usepackage{microtype}
\usepackage[title,titletoc]{appendix}
\usepackage[hidelinks]{hyperref}

\newtheorem{theorem}{Theorem}[section]
\newtheorem{proposition}[theorem]{Proposition}
\newtheorem{lemma}[theorem]{Lemma}
\newtheorem{corollary}[theorem]{Corollary}
\newtheorem{conjecture}[theorem]{Conjecture}
\newtheorem*{theoremA}{Theorem A}
\newtheorem*{theoremB}{Theorem B}
\newtheorem*{theoremC}{Theorem C}
\theoremstyle{remark}
\newtheorem{remark}[theorem]{Remark}

\DeclareMathOperator{\diag}{diag}
\DeclareMathOperator{\rank}{rank}
\newcommand{\R}{\mathbb{R}}
\newcommand{\Z}{\mathbb{Z}}
\newcommand{\M}{\mathcal M}
\newcommand{\one}{\mathbf 1}
\newcommand{\ph}{\varphi}
\allowdisplaybreaks[1]
\hypersetup{pdftitle={Morley Simplices in Higher Dimensions},pdfauthor={Quang Hung Tran}}

\title{Morley Simplices in Higher Dimensions:\\Regularity, Reflections, and Counterexamples}
\author{Quang Hung Tran}
\date{}

\begin{document}
\maketitle

\begin{abstract}
Trisecting the dihedral angles of an $n$-simplex defines its Morley simplex. We study the original simplices for which this simplex is regular. A criterion in terms of the Gram matrix of the facet normals reduces the problem to a matrix equation. The derivative of the Morley map at the regular simplex has two explicit eigenvalues, both nonzero for $n\ge3$; thus the regular simplex is an isolated solution. We conjecture that in dimensions four and five a regular Morley simplex forces two hyperplane reflections interchanging disjoint pairs of vertices. We prove that this conclusion fails in every dimension $6\le n\le200$ and in every dimension $n=\binom k2-1$ with $k\ge9$: in these dimensions there are simplices with regular Morley simplex and no hyperplane reflection symmetry. The examples for $8\le n\le200$ have dihedral symmetry of order $2(n+1)$, while the infinite family is based on the Johnson scheme. In dimension four we give exact constructions of two nonregular examples, defined by irreducible polynomials of degrees $18$ and $8$ with Galois groups $S_{18}$ and $S_8$. Neither example is expressible by radicals. The computer assisted existence proofs use exact rational arithmetic and intervals with outward rounding.
\end{abstract}

\noindent
\textbf{Keywords.} Morley's theorem, Morley simplex, dihedral angle, reflection symmetry, Gram matrix, circulant matrix, Johnson scheme, Galois group, interval arithmetic.

\medskip
\noindent
\textbf{2020 Mathematics Subject Classification.}
Primary 51M04; Secondary 51M20, 52B11, 12F10, 65G30.

\tableofcontents

\section{Introduction}

Morley's trisector theorem says that the adjacent angle trisectors of a triangle meet in the vertices of an equilateral triangle \cite{Morley}; see \cite{Child,CoxeterGreitzer,GrinbergOrhon,Hashimoto,OakleyBaker,TranSynthetic} for proofs and \cite{Braude,DergiadesTran} for extensions. Svrtan and Veljan considered a three dimensional analogue \cite{SvrtanVeljan}. In \cite{TranActa} the author defined the Morley tetrahedron by trisecting the six dihedral angles of a tetrahedron and proposed two converse conjectures.

We call a tetrahedron \emph{isosceles} when its three pairs of opposite edges have equal lengths; regular tetrahedra are included. Definition~1 of \cite{TranActa} excluded the regular case, but the proof of the direct theorem establishes preservation of the opposite edge equalities. We use this statement and the inclusive convention of the companion paper \cite{TranTetra}.

They do not. In the companion paper \cite{TranTetra} we found a nonisosceles tetrahedron $T_1$ and an isosceles but nonregular tetrahedron $T_2$ whose Morley tetrahedra are regular. Both have a pair of opposite edges such that the four remaining edges are equal; in other words, each has two reflection symmetries that interchange disjoint pairs of vertices. We conjectured in \cite{TranTetra} that a regular Morley tetrahedron always forces this, and proved it for every tetrahedron with a nontrivial symmetry.

The Morley construction makes sense in every dimension. For each facet of an $n$-simplex, take the $n$ hyperplanes that trisect the dihedral angles along the ridges of that facet, on the side of the facet. They meet in one point, and the $n+1$ points obtained in this way are the vertices of the \emph{Morley simplex}. For $n\ge3$ let $(G_n)$ be the statement:
\begin{quote}
every $n$-simplex whose Morley simplex is regular has two reflection symmetries that interchange disjoint pairs of vertices.
\end{quote}
The main conjecture of \cite{TranTetra} is $(G_3)$. This paper is about $(G_n)$ for $n\ge4$, and more generally about the simplices with regular Morley simplex.

\subsection*{Results}

The first part of the paper sets up the tools. The Gram matrix of the unit normals of the Morley simplex is an explicit function of the Gram matrix $H$ of the unit normals of the original simplex (Proposition~\ref{sx:gram}). Regularity of the Morley simplex becomes the matrix equation \eqref{sx:regularityeq}, which relates $H$ to the matrix $K=(2\cos(\theta_{ij}/3))$ built from the trisected dihedral angles (Theorem~\ref{sx:regularity}). Section~\ref{sx:reflections} gives a criterion for reflection symmetries. The first main result concerns the regular simplex.

\begin{theoremA}[Theorem~\ref{lin:main}]
Let $n\ge3$. There is a neighbourhood of the regular $n$-simplex in the space of $n$-simplices up to similarity in which the regular simplex is the only simplex with regular Morley simplex. The derivative of the Morley map at the regular simplex has exactly two eigenvalues,
\[
\begin{gathered}
\lambda_1=-\frac{(n-2)(q_n-1)}{3n(1+nq_n)},\; \lambda_2=\frac{q_n-1}{3(q_n+1)},\\
q_n=2\cos\Bigl(\frac13\arccos\frac1n\Bigr).
\end{gathered}
\]
\end{theoremA}

For $n=2$ the first eigenvalue is zero, as it must be by Morley's theorem, and for $n=3$ the two values are the ones found in \cite{TranTetra}.

The second part states the conjectures. In dimensions four and five, the solutions found in our numerical searches had two reflections interchanging disjoint pairs of vertices. This suggests $(G_4)$ and $(G_5)$ (Conjecture~\ref{ld:two-reflections}). The examples in dimensions six and seven show that the same conclusion cannot hold in general. We conjecture that every dimension $n\ge6$ admits a simplex with regular Morley simplex and no hyperplane reflection symmetry (Conjecture~\ref{hd:all-dimensions}). The third part gives the constructions and their proofs.

\begin{theoremB}[Theorems~\ref{hd:counterexample}, \ref{hd:seven-counterexample}, \ref{dh:main} and~\ref{jn:existence}]
Suppose that $6\le n\le200$, or that $n=\binom k2-1$ for some $k\ge9$. Then there is an $n$-simplex whose Morley simplex is regular and which has no hyperplane reflection symmetry. In particular $(G_n)$ is false for these $n$.
\end{theoremB}

In dimensions six and seven the examples have an isometry group of order two. For $8\le n\le200$ they belong to a family of simplices with dihedral symmetry of order $2(n+1)$ (Section~\ref{dh:section}); this family appears to exist for every $n\ge8$ (Conjecture~\ref{dh:conj}), which together with the examples in dimensions six and seven would prove Conjecture~\ref{hd:all-dimensions}. The dimensions $n=\binom k2-1$ come from a two parameter family based on the Johnson scheme (Section~\ref{jn:section}).

In dimension four we know three classes of simplices with regular Morley simplex: the regular simplex and two classes $U_1,U_2$. A cyclic five vertex family contains only the regular simplex (Theorem~\ref{ld:d5}).

\begin{theoremC}[Theorems~\ref{ld:U2}, \ref{ld:U1} and~\ref{ld:galois}]
The classes $U_1$ and $U_2$ are given exactly by a root of an irreducible polynomial of degree $18$, respectively $8$, together with square roots. The Galois groups of the two polynomials are $S_{18}$ and $S_8$, and neither simplex can be expressed by radicals.
\end{theoremC}

Table~\ref{tab:summary} summarizes what is known in each dimension.

\begin{table}[ht]
\centering\small
\begin{tabular}{lp{.72\textwidth}}\toprule
Dimension&Known examples and results\\\midrule
$n=2$&every triangle (Morley's theorem)\\
$n=3$&regular, $T_1$, $T_2$; $(G_3)$ conjectured, proved for tetrahedra with a nontrivial symmetry \cite{TranTetra}\\
$n=4$&regular, $U_1$, $U_2$; $(G_4)$ conjectured (Section~\ref{ld:experiments})\\
$n=5$&numerical evidence for $(G_5)$; the conjecture remains open\\
$n=6,7$&examples without reflection symmetry (Theorems~\ref{hd:counterexample} and~\ref{hd:seven-counterexample})\\
$8\le n\le200$&dihedral examples without reflection symmetry (Theorem~\ref{dh:main})\\
$n=\binom k2-1$, $k\ge9$&Johnson examples without reflection symmetry (Theorem~\ref{jn:existence})\\
other $n\ge201$&open; the smallest open dimension is $201$\\
all $n\ge3$&the regular simplex is isolated (Theorem~\ref{lin:main})\\\bottomrule
\end{tabular}
\caption{Known examples and results. The rows for dimensions three and four do not assert a complete classification.}\label{tab:summary}
\end{table}

\subsection*{On the computations}

Numerical searches suggested the examples and the conjectures. The qualitative observations reported below do not enter any proof or establish a classification. Each computer assisted result is instead reduced to exact algebraic identities and rational interval bounds. Floating point arithmetic supplies starting points and approximate inverses; the checking programs then treat these as rational data and verify the required inequalities exactly. The programs use only the Python standard library and are supplied as ancillary files, with instructions in Appendix~\ref{sec:reproducibility}.

\part{The Morley simplex}
\label{part:framework}

\section{Morley simplices in arbitrary dimension}\label{sx:framework}

Let $n\ge2$, put $N=n+1$, and let $\Delta=A_1\cdots A_N$ be a nondegenerate $n$-simplex in $\R^n$. Let $F_i$ be the facet opposite $A_i$, $S_i$ its $(n-1)$-dimensional volume, $n_i$ its inward unit normal, and $\ell_i$ the signed distance from the hyperplane of $F_i$, positive inside $\Delta$. Let $V$ be the volume of $\Delta$ and $\theta_{ij}$ the interior dihedral angle along the ridge $F_i\cap F_j$. Put $\ph(x)=(3x-x^3)/2$ and
\[
 q_{ii}=1,\; q_{ij}=2\cos(\theta_{ij}/3)\ (i\ne j),\;
 K=(q_{ij}),\; H=(n_i\cdot n_j).
\]
Since $\ph(2\cos x)=-\cos3x$ and $n_i\cdot n_j=-\cos\theta_{ij}$, we have $H_{ij}=\ph(q_{ij})$ for all $i,j$; we write $H=\ph[K]$, with $\ph$ applied entrywise. Note that $1<q_{ij}<2$ for $i\ne j$. The relation $\sum_iS_in_i=0$ gives
\begin{equation}\label{sx:kernel}
 H\succeq0,\; \rank H=n,\; HS=0,\; S=(S_i)>0.
\end{equation}

Fix $i$. On the side of $F_i$, each dihedral angle $\theta_{ij}$, $j\ne i$, is cut by the trisecting hyperplane through $F_i\cap F_j$ that is closer to $F_i$. A point $X$ lies on that hyperplane exactly when $\ell_i(X)\sin(2\theta_{ij}/3)=\ell_j(X)\sin(\theta_{ij}/3)$, that is, when $\ell_j(X)=q_{ij}\ell_i(X)$. So the $n$ hyperplanes meet in the points with $\ell_j(X)=t\,q_{ij}$ for all $j$, where $t=\ell_i(X)$. The identity $\sum_jS_j\ell_j(X)=nV$ fixes $t$, and there is exactly one common point $M_i$, with
\begin{equation}\label{sx:distances}
 \ell_j(M_i)=t_iq_{ij}\;(1\le j\le N),\; t_i=\frac{nV}{(KS)_i}>0.
\end{equation}
Its barycentric coordinates $S_jt_iq_{ij}/(nV)$ are positive, so $M_i$ lies inside $\Delta$. We call $\M(\Delta)=M_1\cdots M_N$ the \emph{Morley simplex} of $\Delta$. For $n=2$ it is Morley's triangle, and for $n=3$ the Morley tetrahedron of \cite{TranActa,TranTetra}.

\begin{lemma}\label{sx:inertia}
The matrix $K$ is negative definite on $\one^\perp$, its inertia is $(1,n)$, and $\operatorname{sign}\det K=(-1)^n$. In particular $K$ is invertible and $\M(\Delta)$ is nondegenerate.
\end{lemma}

\begin{proof}
For $-1\le h\le1$ put $f(h)=2\cos(\frac13\arccos(-h))$. Then $f(1)=1$ and $q_{ij}=f(H_{ij})$ for all $i,j$. We claim that
\[
f(h)=\sqrt3-\sum_{m\ge1}c_mh^m,\; c_m>0,\; c_1=\frac13,
\]
with absolute convergence on $[-1,1]$. Put $y=\sqrt3-f(h)$. The equation $\ph(f(h))=h$ becomes
\[
h=3y\Bigl(1-\frac{y}{\sqrt3}\Bigr)\Bigl(1-\frac{y}{2\sqrt3}\Bigr),
\]
and Lagrange inversion gives
\[
c_m=\frac{1}{m3^m}[u^{m-1}]
\Bigl(1-\frac{u}{\sqrt3}\Bigr)^{-m}
\Bigl(1-\frac{u}{2\sqrt3}\Bigr)^{-m}>0.
\]
The branch is analytic for $|h|<1$, because the only finite branch values of the cubic are $h=\pm1$. Letting $h\uparrow1$ gives $\sum_mc_m=\sqrt3-1$, so the series converges absolutely on $[-1,1]$.

Hence $K=\sqrt3\one\one^T-\sum_{m\ge1}c_mH^{\circ m}$, where $H^{\circ m}$ is the $m$-th Hadamard power. It is the Gram matrix of the tensors $n_i^{\otimes m}$, hence positive semidefinite. For $0\ne z\perp\one$,
\[
z^TKz=-\sum_{m\ge1}c_mz^TH^{\circ m}z\le-\frac13z^THz<0,
\]
because $\ker H=\R S$ with $S>0$, and $z$ is not a multiple of $S$. Since $\one^TK\one>0$, the inertia is $(1,n)$. The matrix of barycentric coordinates of $M_1,\dots,M_N$ is $\diag(t_i)K\diag(S_j)/(nV)$, which is invertible.
\end{proof}

The next proposition expresses the normals of $\M(\Delta)$ through $H$ and $K$ alone.

\begin{proposition}\label{sx:gram}
Let $\mu_1,\dots,\mu_N$ be the barycentric coordinates with respect to $\M(\Delta)$, let $T=\diag(t_i)$ and $B=K^{-1}HK^{-1}$. Then the Gram matrix of the gradients $\nabla\mu_i$ is $T^{-1}BT^{-1}$. Consequently, the Gram matrix of the inward unit normals of $\M(\Delta)$ is
\begin{equation}\label{sx:morleymap}
H_\M=\Bigl(\frac{B_{ij}}{\sqrt{B_{ii}B_{jj}}}\Bigr)_{i,j}.
\end{equation}
\end{proposition}

\begin{proof}
Every affine function $g$ satisfies $g=\sum_ig(M_i)\mu_i$. Applied to $\ell_j$ and combined with \eqref{sx:distances}, this gives $\ell_j=\sum_it_iq_{ij}\mu_i$, and taking gradients, $n_j=\sum_it_iq_{ij}\nabla\mu_i$. Thus the matrix with rows $n_j$ equals $KT$ times the matrix with rows $\nabla\mu_i$. Taking Gram matrices, $H=KT\,G\,TK$, where $G$ is the Gram matrix of the $\nabla\mu_i$, and $G=T^{-1}BT^{-1}$ follows from Lemma~\ref{sx:inertia}. The inward unit normal of the facet of $\M(\Delta)$ opposite $M_i$ is $\nabla\mu_i/|\nabla\mu_i|$; since $T$ is diagonal with positive entries, normalizing $G$ gives \eqref{sx:morleymap}.
\end{proof}

A simplex is determined up to similarity by the Gram matrix of its unit normals, and it is regular exactly when all off-diagonal entries of that matrix equal $-1/n$. This gives the regularity criterion.

\begin{theorem}[Regularity criterion]\label{sx:regularity}
Let $\Delta$ be an $n$-simplex and $B=K^{-1}HK^{-1}$. The following are equivalent:
\begin{enumerate}[label=\textup{(\alph*)}]
\item $\M(\Delta)$ is regular;
\item $B_{ij}=-\frac1n\sqrt{B_{ii}B_{jj}}$ for all $i\ne j$;
\item there is a vector $d=(d_i)$ with positive entries such that, with $D=\diag(d_i)$,
\begin{equation}\label{sx:regularityeq}
 H=4K\Bigl(D^2-\frac1Ndd^T\Bigr)K.
\end{equation}
\end{enumerate}
If $\M(\Delta)$ is regular with edge length $L$, then \eqref{sx:regularityeq} holds with $d_i=t_i/(\sqrt2L)$. Conversely, let $K$ be a real symmetric $N\times N$ matrix with unit diagonal and $1<q_{ij}<2$ for $i\ne j$, such that $H=\ph[K]$ satisfies \eqref{sx:kernel} for some $S$. Then $H$ is the normal Gram matrix of an $n$-simplex $\Delta$, unique up to similarity, whose dihedral angles satisfy $q_{ij}=2\cos(\theta_{ij}/3)$. If moreover \eqref{sx:regularityeq} holds for some $d$ with positive entries, then $\M(\Delta)$ is regular.
\end{theorem}

\begin{proof}
By Proposition~\ref{sx:gram}, (a) and (b) are equivalent. Put $P=I-\frac1N\one\one^T$. If (b) holds, let $e_i=\sqrt{B_{ii}}$ and $E=\diag(e_i)$; then $B=\frac Nn EPE$, and \eqref{sx:regularityeq} follows with $d=\frac12\sqrt{N/n}\,e$, because $DPD=D^2-dd^T/N$. If (c) holds, then $B=4DPD$, so $B_{ii}=4nd_i^2/N$ and $B_{ij}=-4d_id_j/N$, which is (b).

Let $\M(\Delta)$ be regular with edge $L$. The barycentric coordinates of a regular simplex with edge $L$ and centroid $0$ are $\mu_i(X)=\frac1N+\frac2{L^2}m_i\cdot X$, where $m_i$ are the vertices; indeed $m_i\cdot m_j=\frac{L^2}2(\delta_{ij}-\frac1N)$. So the Gram matrix of the gradients is $\frac2{L^2}P$, and Proposition~\ref{sx:gram} gives $B=\frac2{L^2}TPT=4DPD$ with $d_i=t_i/(\sqrt2L)$.

For the converse, let $n_1,\dots,n_N$ be unit vectors in $\R^n$ with Gram matrix $H$. They span $\R^n$ and satisfy $\sum_iS_in_i=0$ with $S>0$, so the set $\{X:n_i\cdot X\ge-1\text{ for all }i\}$ is a nondegenerate simplex with inward unit normals $n_i$. Its shape is unique, since translations and a common dilation account for all choices of the $N$ support numbers. Its dihedral angles satisfy $\cos\theta_{ij}=-H_{ij}=-\ph(q_{ij})=\cos\bigl(3\arccos(q_{ij}/2)\bigr)$, and $3\arccos(q_{ij}/2)\in(0,\pi)$ because $1<q_{ij}<2$. Hence $q_{ij}=2\cos(\theta_{ij}/3)$, and $K$ is the matrix of $\Delta$. By (c)$\Rightarrow$(a), $\M(\Delta)$ is regular.
\end{proof}

In the converse part, the vector $S$ is forced: \eqref{sx:regularityeq} shows that $H$ annihilates $K^{-1}D^{-1}\one$, so $S$ is a positive multiple of $K^{-1}(1/d_i)_i$. In the existence proofs below we either check the positivity of this vector directly or obtain the kernel of $H$ from a symmetry.

\section{Reflections and a symmetry theorem}\label{sx:reflections}

\begin{proposition}\label{sx:reflectiontest}
A hyperplane reflection that maps $\Delta$ to itself interchanges exactly one pair of vertices and fixes the others. The transposition of $A_i$ and $A_j$ is induced by a reflection of $\Delta$ if and only if $|A_iA_k|=|A_jA_k|$ for every $k\notin\{i,j\}$, and this happens if and only if $K$ is invariant under the transposition. More generally, a permutation of the vertices is induced by an isometry of $\Delta$ if and only if it preserves $K$.
\end{proposition}

\begin{proof}
A reflection $\rho$ of $\Delta$ permutes the vertices and is not the identity, so it moves some vertex. If it interchanged two disjoint pairs $A_i,A_j$ and $A_k,A_l$, the vectors $A_i-A_j$ and $A_k-A_l$ would both lie in the $(-1)$-eigenspace of the linear part of $\rho$, which has dimension one; they are linearly independent because the vertices are affinely independent. Hence $\rho$ induces a transposition.

The perpendicular bisector of $A_iA_j$ contains every other vertex exactly under the stated condition on edge lengths, and then the reflection in it induces the transposition. A permutation induced by an isometry preserves dihedral angles and hence $K$. Conversely, a permutation preserving $K$ preserves $H=\ph[K]$, so it is induced by a similarity of $\Delta$ onto itself (Theorem~\ref{sx:regularity}); a similarity that permutes a finite set spanning $\R^n$ affinely fixes its centroid and preserves the sum of the squared distances to it, so it is an isometry.
\end{proof}

\begin{proposition}[Reflection defects]\label{sx:defects}
If $\M(\Delta)$ is regular with edge $L$, then for $i\ne j$
\begin{equation}\label{sx:defecteq}
 \varepsilon_{ij}=q_{ij}+2-\frac{(t_i+t_j)^2}{L^2}\ge0,
\end{equation}
and equality holds exactly when $\Delta$ has the reflection that interchanges $A_i$ and $A_j$. Moreover
\[
 \frac{(t_i+t_j)^2}{L^2(q_{ij}+2)}+\frac{(t_i-t_j)^2}{L^2(2-q_{ij})}\le1.
\]
\end{proposition}

\begin{proof}
Let $v=M_j-M_i$ and $q=q_{ij}$. By \eqref{sx:distances},
\[
 (n_i-n_j)\cdot v=(q-1)(t_i+t_j),\; (n_i+n_j)\cdot v=(q+1)(t_j-t_i).
\]
The vectors $n_i-n_j$ and $n_i+n_j$ are orthogonal, with squared lengths $2-2\ph(q)=(q-1)^2(q+2)$ and $2+2\ph(q)=(2-q)(q+1)^2$. The sum of the squares of the projections of $v$ on these two directions is at most $|v|^2=L^2$, which is the second inequality; the first follows from it. If $\varepsilon_{ij}=0$, then $t_i=t_j$ and $v$ is parallel to $n_i-n_j$. Every other vertex $M_k$ of the regular simplex $\M(\Delta)$ lies on the perpendicular bisector of $M_iM_j$, which is now the hyperplane $\ell_i=\ell_j$. By \eqref{sx:distances}, $t_kq_{ki}=t_kq_{kj}$, so $q_{ik}=q_{jk}$ for all $k\notin\{i,j\}$, and Proposition~\ref{sx:reflectiontest} applies. Conversely, if the reflection exists, it fixes $\M(\Delta)$ and interchanges $M_i$ and $M_j$; then $t_i=t_j$, $v$ is orthogonal to the mirror and so parallel to $n_i-n_j$, and equality holds.
\end{proof}

\begin{corollary}\label{sx:reflectiondistance}
Suppose that $\M(\Delta)$ is regular with edge length $L$. If $\Delta$ has the reflection that interchanges $A_i$ and $A_j$, then
\[
t_i=t_j=L\cos\frac{\theta_{ij}}6,\;\text{equivalently}\; d_i^2=d_j^2=\frac{q_{ij}+2}8.
\]
\end{corollary}

\begin{proof}
Now $\varepsilon_{ij}=0$ and $t_i=t_j$, so $4t_i^2=L^2(q_{ij}+2)=4L^2\cos^2(\theta_{ij}/6)$.
\end{proof}

In the regularity equations this is a linear relation between $d_i^2$ and $q_{ij}$. It explains the identities $8d_i^2=q_{ij}+2$ that appear in \cite{TranTetra} and in Section~\ref{ld:exact}.

In terms of the defects, $(G_n)$ says that the vertices can be labelled so that $\varepsilon_{12}=\varepsilon_{34}=0$; then the perpendicular bisectors of $A_1A_2$ and $A_3A_4$ contain all the other vertices, and they are orthogonal to each other because $A_1-A_2$ and $A_3-A_4$ are. The following theorem is a case in which $(G_n)$ can be proved.

\begin{theorem}[Klein four symmetry on four vertices]\label{sx:v4}
Let $n\ge3$. Suppose that $A_1A_2A_3A_4$ is an isosceles tetrahedron and that every other vertex of $\Delta$ is equidistant from $A_1,\dots,A_4$. If $\M(\Delta)$ is regular, then $\Delta$ has two reflections that interchange disjoint pairs of vertices among $A_1,\dots,A_4$.
\end{theorem}

\begin{proof}
The three double transpositions of $A_1,\dots,A_4$ preserve all edge lengths and fix the other vertices. Hence
\[
 q_{12}=q_{34}=x,\; q_{13}=q_{24}=y,\; q_{14}=q_{23}=z,\; d_1=d_2=d_3=d_4=\delta>0,
\]
and $q_{1j}=q_{2j}=q_{3j}=q_{4j}$ for $j>4$. The vector $(1,1,-1,-1,0,\dots,0)$ and the two vectors obtained from it by permuting the first four coordinates are common eigenvectors of $K$ and $H$, and they are orthogonal to $d$. Put $\Sigma=1+x+y+z$ and $\Lambda=1+\ph(x)+\ph(y)+\ph(z)$. The eigenvalues of $K$ on these vectors are $2u+2-\Sigma$, and those of $H$ are $2+3u-u^3-\Lambda$, for $u=x,y,z$. Equation \eqref{sx:regularityeq} gives
\[
 u^3-3u-2+\Lambda+4\delta^2(2u+2-\Sigma)^2=0,\; u\in\{x,y,z\}.
\]
If $x,y,z$ were distinct, they would be the three roots of this cubic in $u$, and the coefficient of $u^2$ would give $x+y+z=-16\delta^2<0$, which is impossible. So two of them coincide, say $x=y$. Then $K$ is invariant under the transpositions $(1\,4)$ and $(2\,3)$, and Proposition~\ref{sx:reflectiontest} completes the proof.
\end{proof}

\section{Near the regular simplex}\label{lin:section}

Let $\Sigma_n$ be the set of normal Gram matrices of $n$-simplices, that is, of symmetric $N\times N$ matrices $H$ with unit diagonal that satisfy \eqref{sx:kernel}. By Theorem~\ref{sx:regularity}, $\Sigma_n$ is the space of $n$-simplices up to similarity, and Proposition~\ref{sx:gram} describes the Morley construction as the map
\begin{equation}\label{lin:map}
\M(H)=\Bigl(\frac{B_{ij}}{\sqrt{B_{ii}B_{jj}}}\Bigr),\; B=K^{-1}HK^{-1},\; K=f[H],
\end{equation}
where $f(h)=2\cos(\frac13\arccos(-h))$ as in the proof of Lemma~\ref{sx:inertia}. The regular simplex corresponds to $H_0$ with off-diagonal entries $-1/n$, and $\M(H_0)=H_0$. Let
\[
q_n=f(-1/n)=2\cos\Bigl(\frac13\arccos\frac1n\Bigr),
\]
so that $\sqrt3<q_n<2$ for $n\ge2$. Consider the following subspaces of the space of symmetric $N\times N$ matrices with zero diagonal:
\[
W_1=\bigl\{x\one^T+\one x^T-2\diag(x):\ x\in\R^N,\ \textstyle\sum_ix_i=0\bigr\},\;
W_2=\{Z:\ Z\one=0\}.
\]
They have dimensions $n$ and $N(N-1)/2-N=(n+1)(n-2)/2$.

\begin{theorem}\label{lin:main}
Let $n\ge2$. Near $H_0$, the set $\Sigma_n$ is a smooth hypersurface in the affine space of symmetric matrices with unit diagonal, with tangent space $W_1\oplus W_2$ at $H_0$. The map \eqref{lin:map} is smooth near $H_0$, maps $\Sigma_n$ into itself, and its derivative at $H_0$ acts on $W_1$ and $W_2$ as multiplication by
\begin{equation}\label{lin:eigen}
\lambda_1=-\frac{(n-2)(q_n-1)}{3n(1+nq_n)},\;\lambda_2=\frac{q_n-1}{3(q_n+1)},
\end{equation}
respectively. Consequently, for $n\ge3$ the regular simplex has a neighbourhood in $\Sigma_n$ that contains no other simplex with regular Morley simplex.
\end{theorem}

\begin{proof}
For $H$ near $H_0$, the smallest eigenvalue $\mu(H)$ of $H$ is simple, with eigenvector close to $\one$, and the other eigenvalues are close to $N/n$. So near $H_0$ the set $\Sigma_n$ is the zero set of the smooth function $\mu$, and the kernel vector is automatically positive. The derivative of $\mu$ at $H_0$ is $Z\mapsto\one^TZ\one/N$, which is not zero, so $\Sigma_n$ is a smooth hypersurface with tangent space $\{Z:\one^TZ\one=0\}$. This space is $W_1\oplus W_2$: given $Z$ with zero diagonal and $\one^TZ\one=0$, put $x=Z\one/(N-2)$; then $\sum_ix_i=0$, and $X=x\one^T+\one x^T-2\diag(x)$ satisfies $X\one=(N-2)x=Z\one$, so $Z-X\in W_2$. The sum is direct because $X\one=0$ forces $x=0$. The map \eqref{lin:map} is smooth near $H_0$, since $f$ is smooth on $(-1,1)$, $K_0=f[H_0]$ is invertible and $B_{ii}>0$; it maps $\Sigma_n$ to $\Sigma_n$ because $\M(\Delta)$ is a simplex.

Write $q=q_n$, $P=I-\frac1N\one\one^T$ and $\Pi=\frac1N\one\one^T$. Then $H_0=\frac NnP$ and $K_0=(1-q)P+\lambda\Pi$ with $\lambda=1+nq$, so
\[
K_0^{-1}=\frac{P}{1-q}+\frac{\Pi}{\lambda}=\alpha I+\beta\one\one^T,\;
\alpha=\frac1{1-q},\;\beta=\frac1N\Bigl(\frac1\lambda-\alpha\Bigr),
\]
and $B_0=K_0^{-1}H_0K_0^{-1}=cP$ with $c=N/(n(1-q)^2)$. The diagonal entries of $B_0$ are $1/(1-q)^2$ and the off-diagonal ones are $-1/(n(1-q)^2)$. Let $Z$ be tangent and put $\eta=f'(-1/n)=1/\ph'(q)=2/(3(1-q^2))$. The derivative of $K$ is $\eta Z$, so
\begin{equation}\label{lin:dB}
\delta B=K_0^{-1}ZK_0^{-1}-\eta\bigl(K_0^{-1}ZB_0+B_0ZK_0^{-1}\bigr),
\end{equation}
and differentiating \eqref{lin:map} at $B_0$ gives
\begin{equation}\label{lin:dM}
\delta\M_{ij}=(1-q)^2\Bigl(\delta B_{ij}+\frac{\delta B_{ii}+\delta B_{jj}}{2n}\Bigr)\;(i\ne j).
\end{equation}
Since $\ph(q)=-1/n$ and $1-\ph(q)=\frac12(q-1)^2(q+2)$, we have $(q-1)^2(q+2)=2N/n$, and therefore
\begin{equation}\label{lin:etac}
(1-q)\eta c=\frac{2N}{3n(1+q)(1-q)^2}=\frac{q+2}{3(q+1)}.
\end{equation}

Let $Z\in W_2$. Then $PZ=ZP=Z$ and $\Pi Z=0$, so $K_0^{-1}Z=\alpha Z$ and $B_0Z=cZ$. By \eqref{lin:dB}, $\delta B=(\alpha^2-2\eta c\alpha)Z$, which has zero diagonal, and by \eqref{lin:dM} and \eqref{lin:etac}
\[
\delta\M=\bigl(1-2(1-q)\eta c\bigr)Z=\Bigl(1-\frac{2(q+2)}{3(q+1)}\Bigr)Z=\lambda_2Z.
\]

Let $Z=X=U-2D_x\in W_1$, where $U=x\one^T+\one x^T$, $D_x=\diag(x)$ and $\sum_ix_i=0$. Using $\one^Tx=0$, $X\one=(N-2)x$ and $K_0^{-1}x=\alpha x$, $K_0^{-1}\one=\one/\lambda$, one finds
\[
K_0^{-1}UK_0^{-1}=\frac\alpha\lambda U,\; K_0^{-1}D_xK_0^{-1}=\alpha^2D_x+\alpha\beta U,
\]
\[
K_0^{-1}XB_0+B_0XK_0^{-1}=c\Bigl(2\alpha X+(N-2)\Bigl(\beta-\frac\alpha N\Bigr)U\Bigr).
\]
Substituting in \eqref{lin:dB} gives $\delta B=\gamma U-2\alpha(\alpha-2\eta c)D_x$ with
\[
\gamma=\frac{(N-2)(\alpha-\eta c)}{N\lambda}+\frac{2\alpha(\alpha-2\eta c)}{N}.
\]
Hence $\delta B_{ij}=\gamma(x_i+x_j)$ for $i\ne j$ and $\delta B_{ii}=(2\gamma-2\alpha(\alpha-2\eta c))x_i$, and \eqref{lin:dM} gives $\delta\M=\lambda_1'X$ with
\[
\lambda_1'=\frac{(1-q)^2}{n}\Bigl(\frac{(N-2)(\alpha-\eta c)}{\lambda}+\alpha(\alpha-2\eta c)\Bigr).
\]
By the computation for $W_2$, $(1-q)^2\alpha(\alpha-2\eta c)=\lambda_2$, and by \eqref{lin:etac}, $(1-q)^2(\alpha-\eta c)=(1-q)(2q+1)/(3(q+1))$. Therefore
\[
\lambda_1'=\frac{q-1}{3n(q+1)}\Bigl(1-\frac{(N-2)(2q+1)}{1+(N-1)q}\Bigr)=-\frac{(N-3)(q-1)}{3n(1+nq)}=\lambda_1.
\]

For $n\ge3$ both $\lambda_1$ and $\lambda_2$ are nonzero, so the derivative of $\M$ restricted to $\Sigma_n$ is invertible at $H_0$. By the inverse function theorem, $\M$ is injective on a neighbourhood of $H_0$ in $\Sigma_n$, and $\M(H)=H_0=\M(H_0)$ forces $H=H_0$ there.
\end{proof}

\begin{remark}\label{lin:remarks}
For $n=2$ we have $W_2=0$ and $\lambda_1=0$, so the derivative vanishes, in agreement with Morley's theorem, which says that $\M$ is constant on $\Sigma_2$. For $n=3$, $3q^3-9q-2=0$ and
\[
\lambda_1=-\frac{q-1}{9(1+3q)}\approx-0.014252,\;\lambda_2=\frac{q-1}{3(q+1)}\approx0.098093;
\]
these are the numbers $\kappa_3$ and $\kappa_2$ in the proof of \cite[Theorem 13.1]{TranTetra}, which uses different coordinates on the space of tetrahedra. For the next dimensions,
\[
\begin{array}{crrrrrr}
n&4&5&6&7&8&12\\\hline
\lambda_1&-.016384&-.015942&-.014899&-.013775&-.012720&-.009538\\
\lambda_2&.096089&.094831&.093968&.093339&.092859&.091716
\end{array}
\]
As $n\to\infty$, $q_n\to\sqrt3$, so $\lambda_1\sim-(\sqrt3-1)/(3\sqrt3\,n)$ and $\lambda_2\to(2-\sqrt3)/3\approx0.0893$. The theorem gives a neighbourhood for each $n$ but says nothing about its size.
\end{remark}

\part{Conjectures in higher dimensions}
\label{part:hdconjectures}

\section{The conjectures}\label{hd:conjectures}

In dimension three all known simplices with regular Morley simplex have two reflections on disjoint pairs of vertices, and this is the content of the main conjecture of \cite{TranTetra}. Our experiments, described in Part~\ref{part:hdresults}, suggest that this persists in dimensions four and five and breaks down in dimension six.

\begin{conjecture}[Two reflections in dimensions four and five]
\label{ld:two-reflections}
Let $n\in\{4,5\}$. If the Morley simplex of an $n$-simplex $\Delta$ is regular, then $\Delta$ has two hyperplane reflection symmetries that interchange disjoint pairs of vertices.
\end{conjecture}

By Proposition~\ref{sx:reflectiontest} each of the two reflections fixes all vertices outside its pair, and the two mirrors are perpendicular. The conjecture is the statement $(G_4)$ and $(G_5)$; it does not claim a list of all similarity classes.

\begin{conjecture}[No reflections from dimension six on]
\label{hd:all-dimensions}
For every $n\ge6$ there is an $n$-simplex with no hyperplane reflection symmetry whose Morley simplex is regular.
\end{conjecture}

This is stronger than the failure of $(G_n)$, since it asks for simplices without any reflection. Theorem~B proves it for $6\le n\le200$ and for $n=\binom k2-1$, $k\ge9$. The dihedral family of Section~\ref{dh:section} suggests a way to prove it in general (Conjecture~\ref{dh:conj}).

\part{Results, possible approaches, and open problems}
\label{part:hdresults}

Section~\ref{ld:experiments} deals with dimensions four and five, including the exact form of the two exceptional classes in dimension four. Sections~\ref{hd:six} and~\ref{hd:seven} give the examples in dimensions six and seven, Section~\ref{dh:section} the dihedral family and Section~\ref{jn:section} the Johnson family. Sections~\ref{hd:directions} and~\ref{hd:open} discuss the remaining dimensions and list open problems.

\section{Dimensions four and five}\label{ld:experiments}

In dimension four, Conjecture~\ref{ld:two-reflections} says that the vertices can be labelled so that
\[
\begin{gathered}
|A_1A_3|=|A_1A_4|=|A_2A_3|=|A_2A_4|,\\
|A_1A_5|=|A_2A_5|,\; |A_3A_5|=|A_4A_5|.
\end{gathered}
\]
The fifth vertex then lies on both mirrors.

The searches in dimension four led to the regular simplex and two nonregular examples, denoted by $U_1$ and $U_2$. Their existence is proved in Section~\ref{ld:exact}. In $U_1$, three vertices form an equilateral triangle of side $1$, the other two vertices are at distance $2.65171181$, and the six edges joining the two groups have length $1.7483635$; the full symmetry group is $S_3\times S_2$. In $U_2$, an apex is at distance $1$ from the four vertices of an isosceles tetrahedron, with one pair of opposite edges of length $1.56445438$ and four edges of length $1.31143372$; its full symmetry group has order $8$. The dihedral angles are approximately $67.615^\circ$, $76.911^\circ$, $106.914^\circ$ for $U_1$ and $63.012^\circ$, $76.551^\circ$, $106.793^\circ$ for $U_2$. Both examples have two reflections interchanging disjoint pairs of vertices.

These examples do not settle the classification in dimension four. Two symmetric families can, however, be handled directly: Theorem~\ref{sx:v4} treats the Klein four family, and the following theorem treats a cyclic action on all five vertices.

\begin{theorem}[The cyclic five vertex family]\label{ld:d5}
Let $\Delta$ be a four-simplex with vertices indexed modulo five, such that all edges $A_iA_{i+1}$ have the same length and all edges $A_iA_{i+2}$ have the same length. If $\M(\Delta)$ is regular, then $\Delta$ is regular.
\end{theorem}

\begin{proof}
The rotation $i\mapsto i+1$ is a symmetry of $\Delta$, so $K$ is circulant with first row $(1,p,r,r,p)$ and the $d_i$ are equal. Put $a=\ph(p)$. Since $H$ annihilates the vector $\one$ of facet volumes, $\ph(r)=-\frac12-a$. The nonconstant real Fourier subspaces have dimension two, and on them $H$ has the eigenvalues
\[
 h_1=\frac{5+\sqrt5}{4}+\sqrt5\,a,\; h_2=\frac{5-\sqrt5}{4}-\sqrt5\,a,
\]
and $K$ has the eigenvalues
\[
 \kappa_1=1+2p\cos72^\circ+2r\cos144^\circ,\; \kappa_2=1+2p\cos144^\circ+2r\cos72^\circ.
\]
Positivity of $h_1,h_2$ gives $-(1+\sqrt5)/4<a<(\sqrt5-1)/4$. By \eqref{sx:regularityeq} with $d$ constant, $h_k=4d_1^2\kappa_k^2$ for $k=1,2$, so regularity is equivalent to
\begin{equation}\label{ld:g}
 g(a)=h_1\kappa_2^2-h_2\kappa_1^2=0.
\end{equation}
Interchanging $p$ and $r$ shows $g(-\frac12-a)=-g(a)$, so it suffices to consider $a\le-\frac14$. Rational interval arithmetic, with adaptive bisection, gives
\[
 g<0\ \text{on}\ [-809017/10^6,\,-51/200],\; g'>0\ \text{on}\ [-51/200,\,-1/4].
\]
The first left endpoint is smaller than $-(1+\sqrt5)/4$. Here $\sqrt5$ and the inverse of $\ph$ are enclosed by rational bisection, and the derivative uses $dp/da=1/\ph'(p)$ and $dr/da=-1/\ph'(r)$. The first inequality needs $61$ subintervals and the second one interval. Since $g(-1/4)=0$, this is the only zero of $g$ with $a\le-1/4$, and by the symmetry the only zero overall. Hence $p=r$, all edges are equal, and $\Delta$ is regular.
\end{proof}

The numerical searches in dimension five also produced solutions with two reflections on disjoint pairs of vertices. Some approximations appeared to have no third disjoint transposition symmetry. We have not certified these observations, and they provide no classification in dimension five. Conjecture~\ref{ld:two-reflections} asks only for two reflections.

\subsection{The exact form of \texorpdfstring{$U_1$ and $U_2$}{U1 and U2}}\label{ld:exact}

Both classes have two reflections on disjoint pairs, and their symmetry reduces the regularity criterion to a few equations, which we can solve exactly. By Proposition~\ref{sx:reflectiontest} the entries $q_{ij}$ are constant on the orbits of the symmetry group on pairs of vertices, and the $d_i$ are constant on its orbits on vertices. For $N=5$ we write \eqref{sx:regularityeq} as
\begin{equation}\label{ld:system}
F_{ij}=4\sum_{k=1}^5d_k^2q_{ik}q_{jk}-\frac45u_iu_j-\ph(q_{ij})=0\;(i\le j),\; u=Kd.
\end{equation}

\begin{theorem}[The class $U_2$]\label{ld:U2}
Let $p\approx1.6262490796$ be the unique positive root of
\[
\begin{split}
P_8(t)={}&436t^8+4404t^7+11348t^6-144840t^5-143874t^4\\
&+911670t^3+843917t^2-1471902t-1469159,
\end{split}
\]
and put
\[
q=\frac{p+4-\sqrt{3(4-p^2)}}{2},\; \delta=\sqrt{\frac{p+2}8}.
\]
Let $W=W(p)$, $E=E(p)$ and $\rho=\rho(p)$, where $W,E,\rho\in\mathbb Q[t]$ are the polynomials of degree seven given in Appendix~\ref{app:U2}. Then $3<W<4$, $E>0$, $\rho>0$ and $W\rho^2=\delta^2E$. Put $w=\sqrt W$ and $\varepsilon=\sqrt E$, let $K$ have the entries
\[
\begin{gathered}
q_{12}=q_{34}=p,\; q_{13}=q_{14}=q_{23}=q_{24}=q,\\
q_{i5}=w\; (1\le i\le4).
\end{gathered}
\]
and let $d=(\delta,\delta,\delta,\delta,\varepsilon)$. Then $K$ and $d$ satisfy \eqref{ld:system} and define a four-simplex with regular Morley simplex. Its vertex $A_5$ is equidistant from $A_1,\dots,A_4$, it has $|A_1A_2|=|A_3A_4|$, the other four edges of $A_1A_2A_3A_4$ are equal, and it has the reflections $(1\,2)$ and $(3\,4)$. Its full isometry group has order $8$. Numerically $q\approx1.8049155097$ and $w\approx1.8671107974$; this is the class $U_2$.
\end{theorem}

\begin{proof}
The permutations $(1\,2)$, $(3\,4)$ and $(1\,3)(2\,4)$ preserve $K$ and $d$, so the fifteen equations \eqref{ld:system} reduce to five. Put $\Sigma=1+p+2q$, $u_b=\delta\Sigma+w\varepsilon$ and $u_a=4w\delta+\varepsilon$; these are the entries of $u$ at the base vertices and at $A_5$. The five equations are
\begin{align*}
F_{11}&=4\delta^2(1+p^2+2q^2)+4\varepsilon^2W-\tfrac45u_b^2-1,\\
F_{12}&=4\delta^2(2p+2q^2)+4\varepsilon^2W-\tfrac45u_b^2-\ph(p),\\
F_{13}&=4\delta^2(2q+2pq)+4\varepsilon^2W-\tfrac45u_b^2-\ph(q),\\
F_{55}&=16\delta^2W+4\varepsilon^2-\tfrac45u_a^2-1,\\
F_{15}&=4w(\delta^2\Sigma+\varepsilon^2)-\tfrac45u_bu_a-\ph(w).
\end{align*}
First, $F_{11}-F_{12}=4\delta^2(1-p)^2-(1-\ph(p))$ vanishes because $1-\ph(p)=\frac12(1-p)^2(p+2)$; this is Corollary~\ref{sx:reflectiondistance}. Next, $F_{11}-F_{13}=4\delta^2\bigl((1-q)^2+(p-q)^2\bigr)-(1-\ph(q))$, and with $8\delta^2=p+2$ this equals $\frac12(p-q)\bigl[(p+2)(p-q)+(1-q)^2\bigr]$; it vanishes because $q$ is a root of $q^2-(p+4)q+(p+1)^2$. Finally, $\delta\varepsilon=\rho w$ gives
\[
u_b^2=\delta^2\Sigma^2+2\Sigma W\rho+WE,\;
u_a^2=16W\delta^2+8W\rho+E,\;
u_bu_a=w\bigl(4\delta^2\Sigma+E+\rho(\Sigma+4W)\bigr),
\]
and $\ph(w)=w(3-W)/2$. Hence $F_{11}$, $F_{55}$ and $F_{15}/w$ lie in $\mathbb Q(p)$, and so does $q$, which is the polynomial $q(p)$ of Appendix~\ref{app:U2}. A computation modulo $P_8$ shows that these three elements vanish, that $q(p)$ satisfies the quadratic equation above, and that $W\rho^2=\delta^2E$.

A Sturm sequence shows that $P_8$ has exactly one positive root. Rational interval arithmetic at this root gives $1<p<2$, $1<q<2$, $3<W<4$, $E>0$ and $\rho>0$. In particular $1<w<2$, $\varepsilon=\rho w/\delta>0$, and $q$ is the smaller root of the quadratic, as in the statement. The matrix $K$ has the eigenvalues $1-p$ (twice) and $1+p-2q$, and on the span of $\one_{\{1,2,3,4\}}$ and $e_5$ it acts by $\bigl(\begin{smallmatrix}\Sigma&w\\4w&1\end{smallmatrix}\bigr)$; the interval computation shows that these are invertible and that the vector $K^{-1}(1/d_i)_i$ is positive. Since \eqref{ld:system} is \eqref{sx:regularityeq}, the matrix $H$ is positive semidefinite of rank four with a positive kernel vector, and Theorem~\ref{sx:regularity} applies. The same interval computation gives $p<q<w$. A permutation preserving $K$ must therefore fix $A_5$ and preserve the two pairs $\{1,2\}$ and $\{3,4\}$ as an unordered pair of pairs. There are exactly eight such permutations, generated by $(1\,2)$, $(3\,4)$ and $(1\,3)(2\,4)$. These are the full isometry group.
\end{proof}

\begin{theorem}[The class $U_1$]\label{ld:U1}
Let $b\approx1.6254291551$ be the unique root in $(1,2)$ of
\[
\begin{split}
P_{18}(t)={}&8168t^{18}+135252t^{17}+970494t^{16}+2451863t^{15}-5010906t^{14}\\
&-38126877t^{13}-34880165t^{12}+179734110t^{11}+360727044t^{10}\\
&-415345704t^9-1288930653t^8+703616631t^7+2452369115t^6\\
&-1419047031t^5-2449217616t^4+2285215840t^3+643683867t^2\\
&-1537758084t+500360812.
\end{split}
\]
There are polynomials $A,C\in\mathbb Q[t]$ of degree $17$, given in the ancillary file \texttt{check\_U1\_U2.py}, such that $a=A(b)\in(1,2)$ and $C=C(b)\in(1,4)$. Put $c=\sqrt C$, $d_T=\sqrt{(a+2)/8}$ and $d_P=\sqrt{(b+2)/8}$. Let $K$ have the entry $a$ for pairs inside $T=\{1,2,3\}$, $q_{45}=b$, and $c$ for pairs between $T$ and $P=\{4,5\}$, and let $d=(d_T,d_T,d_T,d_P,d_P)$. Then $K$ and $d$ satisfy \eqref{ld:system} and define a four-simplex with regular Morley simplex and symmetry group $S_3\times S_2$. Numerically $a\approx1.8472457767$ and $c\approx1.8031041137$; this is the class $U_1$.
\end{theorem}

\begin{proof}
There are five orbits of pairs $i\le j$. Write $D_T=d_T^2$, $D_P=d_P^2$, $m=cd_Td_P$, $u_T=d_T(1+2a)+2cd_P$ and $u_P=d_P(1+b)+3cd_T$. As before, $F_{11}-F_{12}$ and $F_{44}-F_{45}$ vanish by Corollary~\ref{sx:reflectiondistance}. Using $u_T^2=D_T(1+2a)^2+4(1+2a)m+4CD_P$ and the analogous expansions, the remaining equations $F_{11}$, $F_{44}$ and $cF_{14}$ become
\begin{align*}
&4D_T(1+2a^2)+8CD_P-\tfrac45\bigl(D_T(1+2a)^2+4(1+2a)m+4CD_P\bigr)=1,\\
&4D_P(1+b^2)+12CD_T-\tfrac45\bigl(D_P(1+b)^2+6(1+b)m+9CD_T\bigr)=1,\\
&4C\bigl(D_T(1+2a)+D_P(1+b)\bigr)\\
&-\tfrac45\bigl(m(1+2a)(1+b)+3CD_T(1+2a)+2CD_P(1+b)+6Cm\bigr)=\tfrac{3C-C^2}2,
\end{align*}
together with $m^2=CD_TD_P$. The first equation is linear in $m$ and determines $m\in\mathbb Q(b)$. The other two equations and $m^2=CD_TD_P$ hold exactly modulo $P_{18}$. A Sturm sequence shows that $b$ is the only root of $P_{18}$ in $(1,2)$, and interval arithmetic at $b$ gives $1<b<c<a<2$ and $m>0$. The matrix $K$ has the eigenvalues $1-a$ and $1-b$ and acts by $\bigl(\begin{smallmatrix}1+2a&2c\\3c&1+b\end{smallmatrix}\bigr)$ on the span of $\one_T$ and $\one_P$; interval arithmetic shows that $K$ is invertible and that $K^{-1}(1/d_i)_i$ is positive. Theorem~\ref{sx:regularity} gives the simplex. Since $a$, $b$ and $c$ are distinct, every permutation preserving $K$ preserves the sets $T$ and $P$. Conversely, every permutation within either set preserves $K$. Proposition~\ref{sx:reflectiontest} therefore gives the full isometry group $S_3\times S_2$.
\end{proof}

\begin{theorem}\label{ld:galois}
The polynomials $P_8$ and $P_{18}$ are irreducible over $\mathbb Q$ with Galois groups $S_8$ and $S_{18}$. Consequently $p$ and $b$ cannot be expressed by radicals, and no simplex similar to $U_1$ or $U_2$ has vertex coordinates expressible by radicals.
\end{theorem}

\begin{proof}
By Dedekind's theorem, modulo a prime that divides neither the leading coefficient nor the discriminant, the degrees of the irreducible factors give the cycle type of an element of the Galois group $G$. Modulo $17$, $P_8$ is irreducible; modulo $41$ it factors with degrees $7,1$; modulo $7$ with degrees $3,3,2$. So $P_8$ is irreducible and $G$ contains an $8$-cycle. A $7$-cycle fixing a point makes $G$ doubly transitive, hence primitive, and the cube of an element of type $(3,3,2)$ is a transposition. A primitive group containing a transposition is symmetric, so $G=S_8$.

Modulo $193$, $P_{18}$ is irreducible; modulo $229$ it factors with degrees $17,1$; modulo $7$ with degrees $9,5,3,1$. As before $G$ is primitive and contains an $18$-cycle. The ninth power of an element of type $(9,5,3,1)$ is a $5$-cycle. By Jordan's theorem a primitive group of degree $m$ that contains a cycle of prime length at most $m-3$ contains the alternating group; since an $18$-cycle is odd, $G=S_{18}$.

Neither group is solvable. If a simplex similar to $U_2$ had coordinates in radicals, then so would its unit normals and the numbers $\cos\theta_{ij}=-\ph(q_{ij})$; each $q_{ij}$ is a root of $x^3-3x-2\cos\theta_{ij}$, so $p$ would be expressible by radicals. The same argument applies to $b$ and $U_1$.
\end{proof}

\begin{remark}
Integer relation searches suggested the defining polynomials for $U_1$ and $U_2$; the exact identities above establish the constructions independently of those searches. The edge lengths give further algebraic relations. For example, if the lateral edges of $U_2$ have length one, the squared base edges are approximately $2.4475174946$ and $1.7198583951$, and their ratio $r\approx1.4230924485$ satisfies
\[
\begin{split}
&78181r^8-1176580r^7+7163223r^6-22603834r^5+39354595r^4\\
&-36674496r^3+15630705r^2-1653858r+11664=0.
\end{split}
\]
For $U_1$ with triangle side one, the squared distance between the two remaining vertices is approximately $7.0315755039$ and the squared cross edges are approximately $3.0567749302$; their minimal polynomials have degree $18$. By contrast, the tetrahedron $T_2$ is expressible by square roots, while $T_1$ is defined by a sextic with Galois group $S_6$ \cite{TranTetra}.
\end{remark}

\section{An example in dimension six}\label{hd:six}

The following example has only one nonidentity isometry, which interchanges three pairs of vertices simultaneously.

\begin{theorem}\label{hd:counterexample}
There is a six-simplex with regular Morley simplex whose isometry group has order two and contains no hyperplane reflection. In particular $(G_6)$ is false.
\end{theorem}

The following rounded coordinates in $\R^3\oplus\R^3$ describe the example:
\[
 O=0,\; B_j=(m_j,h_j),\; B'_j=(m_j,-h_j),\; j=1,2,3,
\]
\[
\begin{array}{crrrrrr}
\toprule
j&\multicolumn{3}{c}{m_j}&\multicolumn{3}{c}{h_j}\\\midrule
1&.806238588&0&0&.591590516&0&0\\
2&-.120730325&1.748497240&0&.728537434&1.484244514&0\\
3&-.040254321&-1.580278423&2.591366719&.441900817&-1.560837637&4.631779393\\
\bottomrule
\end{array}
\]
The isometry $(m,h)\mapsto(m,-h)$ fixes $O$ and interchanges $B_j$ and $B'_j$. In this scale the Morley edge is approximately $.119877630$, and the dihedral angles lie between $63.6571^\circ$ and $121.5481^\circ$. If one applies the Morley construction to these nine-digit coordinates, the $21$ Morley edges agree to a relative spread of about $2.3\cdot10^{-11}$. This is only a check; the proof uses the regularity equations.

\begin{proof}
Put $N=7$ and order the vertices $O,B_1,B'_1,B_2,B'_2,B_3,B'_3$. The unknowns are the $21$ numbers $q_{ij}$, $i<j$, in lexicographic order, and $d_1,\dots,d_7$. With $u=Kd$, equation \eqref{sx:regularityeq} consists of the $28$ equations
\begin{equation}\label{hd:system}
 F_{ij}=4\sum_{a=1}^Nd_a^2q_{ia}q_{ja}-\frac4Nu_iu_j-\ph(q_{ij})=0,\; i\le j.
\end{equation}
The permutation $\sigma=(B_1\,B'_1)(B_2\,B'_2)(B_3\,B'_3)$ acts on the unknowns, and the system is invariant under it.

The script \texttt{check\_dim6\_dim7.py} starts from approximations to $20$ decimals, improves them by Newton's method, rounds the result to a point $x_0\in\mathbb Q^{28}$ with $60$ decimals, and replaces $x_0$ by the average of $x_0$ and its image under $\sigma$, so that $x_0$ is invariant under $\sigma$. Let $C$ be a floating point approximate inverse of the Jacobian $DF(x_0)$; its entries are binary fractions, hence rational. Let $\rho=10^{-30}$ and $\mathcal I=x_0+[-\rho,\rho]^{28}$. Exact rational arithmetic, together with outward rounded interval arithmetic for an enclosure $DF(\mathcal I)$ of the Jacobian on $\mathcal I$, gives
\begin{equation}\label{hd:bounds}
 \kappa:=\|I-C\,DF(x_0)\|_\infty+\|C\|_\infty\max_i\sum_jw\bigl(DF(\mathcal I)_{ij}\bigr)<4\cdot10^{-13},\;
 \|CF(x_0)\|_\infty<5\cdot10^{-61},
\end{equation}
where $w$ is the width of an interval. Let $g(x)=x-CF(x)$. For $x,y\in\mathcal I$ we have $g(x)-g(y)=(I-C\bar J)(x-y)$, where $\bar J$ is the mean of $DF$ over the segment from $y$ to $x$; its entries lie in the corresponding intervals of $DF(\mathcal I)$, as do those of $DF(x_0)$. Hence $\|g(x)-g(y)\|_\infty\le\kappa\|x-y\|_\infty$ and $\|g(x)-x_0\|_\infty\le\kappa\rho+\|CF(x_0)\|_\infty<\rho$. So $g$ is a contraction of $\mathcal I$ into itself and has a unique fixed point $x_*$ (see \cite{Krawczyk,Neumaier,Rump} for this kind of argument). The first bound shows that $C\,DF(x_0)$, and therefore $C$, is invertible, so $x_*$ is the unique zero of $F$ in $\mathcal I$. Since $F$, $x_0$ and $\mathcal I$ are invariant under $\sigma$, so is $x_*$.

The script also checks $1<q_{ij}<2$ and $d_i>0$ on $\mathcal I$. Let $K_0=K(x_0)$, $R=K_0^{-1}$ and $S_0=R\,(1/d_{0,i})_i$, computed exactly. On $\mathcal I$ we have $\|K-K_0\|_\infty\le(N-1)\rho$; with $e=(N-1)\rho\|R\|_\infty<1$ and $d_{\min}=\min_id_{0,i}$, the Neumann series gives
\[
 \bigl\|K^{-1}(1/d_i)_i-S_0\bigr\|_\infty\le\frac{\|R\|_\infty}{1-e}\Bigl(\frac{\rho}{d_{\min}(d_{\min}-\rho)}+(N-1)\rho\|S_0\|_\infty\Bigr)<2\cdot10^{-29},
\]
while $\min_i(S_0)_i>0.037$. At $x_*$ we have $H=4KDPDK$ with $P=I-\frac1N\one\one^T$, which is positive semidefinite of rank six with kernel spanned by the positive vector $K^{-1}(1/d_i)_i$. Theorem~\ref{sx:regularity} gives a six-simplex with regular Morley simplex.

Since $x_*$ is invariant under $\sigma$, the permutation $\sigma$ is induced by an isometry (Proposition~\ref{sx:reflectiontest}). For each of the other $5038$ nonidentity permutations $\pi$ of the vertices, the script checks
\[
 \max_{i<j}\bigl|(K_0)_{ij}-(K_0)_{\pi(i)\pi(j)}\bigr|>0.094>2\rho,
\]
so $\pi$ does not preserve $K(x_*)$. The isometry group is therefore $\{\mathrm{id},\sigma\}$. As $\sigma$ interchanges three pairs of vertices, it is not a reflection.
\end{proof}

\section{An example in dimension seven}\label{hd:seven}

The seven dimensional example has an involution interchanging four pairs of vertices. As in dimension six, no individual pair can be interchanged while all other vertices are fixed.

\begin{theorem}\label{hd:seven-counterexample}
There is a seven-simplex with regular Morley simplex whose isometry group has order two and contains no hyperplane reflection. In particular $(G_7)$ is false.
\end{theorem}

Write the vertices as $B_i=(m_i,h_i)$ and $B'_i=(m_i,-h_i)$, $1\le i\le4$, in $\R^3\oplus\R^4$. The following rounded coordinates have $|B_1B'_1|=1$:
\[
\begin{array}{crrr}
\toprule
i&\multicolumn{3}{c}{m_i}\\
\midrule
1&0&0&0\\
2&2.5492916654&0&0\\
3&.1758897556&.6481344734&0\\
4&-.3141249054&.6780199791&.9695952075\\
\bottomrule
\end{array}
\]
\[
\begin{array}{crrrr}
\toprule
i&\multicolumn{4}{c}{h_i}\\
\midrule
1&.5&0&0&0\\
2&-.2953023803&3.6143846851&0&0\\
3&-.0694138393&-.0179817115&.3525622431&0\\
4&-.2877179405&-.2409904414&-.1940594065&.9025294303\\
\bottomrule
\end{array}
\]
The Morley edge is approximately $.076947895588$, and the dihedral angles lie between $65.9621^\circ$ and $115.2717^\circ$.

\begin{proof}
The proof is the same as for Theorem~\ref{hd:counterexample}. Now $N=8$, the vertices are ordered $B_1,B'_1,\dots,B_4,B'_4$, there are $28$ unknowns $q_{ij}$ and eight unknowns $d_i$, and \eqref{hd:system} consists of $36$ equations. The involution is $\sigma=(B_1\,B'_1)(B_2\,B'_2)(B_3\,B'_3)(B_4\,B'_4)$. With $\rho=10^{-30}$ the script \texttt{check\_dim6\_dim7.py} verifies
\[
\kappa<3\cdot10^{-12},\;\|CF(x_0)\|_\infty<5\cdot10^{-61},
\]
$1<q_{ij}<2$ and $d_i>0$ on $\mathcal I$, $\min_i(S_0)_i>0.029$ with an error below $2\cdot10^{-29}$, and
\[
 \max_{i<j}\bigl|(K_0)_{ij}-(K_0)_{\pi(i)\pi(j)}\bigr|>0.083>2\rho
\]
for every permutation $\pi\notin\{\mathrm{id},\sigma\}$ of the eight vertices. The isometry group is $\{\mathrm{id},\sigma\}$, and $\sigma$, which interchanges four pairs, is not a reflection.
\end{proof}

\section{Dihedral simplices}\label{dh:section}

The examples in dimensions six and seven use an involution interchanging several pairs of vertices. A different construction starts in dimension eight: cyclic symmetry reduces the regularity equations to a system with about half as many unknowns as vertices. We certify a solution for every number of vertices from $9$ to $201$.

Let $N\ge5$, index the vertices by $\Z/N$, put $M=\lfloor N/2\rfloor$, and let $|k|\in\{0,\dots,M\}$ be the cyclic distance of $k\in\Z/N$ from $0$. Given $c_1,\dots,c_M$, put $c_0=1$, $r_k=c_{|k|}$ for $k\in\Z/N$, $\sigma=\sum_kr_k$, and let $K$ be the symmetric circulant matrix with entries $q_{ij}=r_{j-i}$. Consider the $M+1$ equations
\begin{equation}\label{dh:eq}
E_m=4T\sum_{k\in\Z/N}r_kr_{k-m}-\frac{4T}{N}\sigma^2-\ph(c_m)=0,\; m=0,1,\dots,M,
\end{equation}
in the $M+1$ unknowns $c_1,\dots,c_M,T$, and the inequalities
\begin{equation}\label{dh:pos}
h_k=\sum_{j\in\Z/N}\ph(r_j)\cos\frac{2\pi jk}{N}>0,\; k=1,\dots,M.
\end{equation}

\begin{lemma}\label{dh:lemma}
Let $N\ge5$, and let $c_1,\dots,c_M\in(1,2)$ and $T>0$ satisfy \eqref{dh:eq} and \eqref{dh:pos}. Then there is an $(N-1)$-simplex with $q_{ij}=c_{|j-i|}$ whose Morley simplex is regular. The dihedral group $D_N$, generated by $i\mapsto i+1$ and $i\mapsto-i$, acts on it by isometries. If $c_1,\dots,c_M$ are distinct, then its isometry group is exactly $D_N$, of order $2N$, and it has no hyperplane reflection symmetry.
\end{lemma}

\begin{proof}
The matrices $H=\ph[K]$ and $4TK(I-\frac1N\one\one^T)K$ are symmetric circulants. Since $K\one=\sigma\one$, the entry of the second matrix in position $(0,m)$ is $4T\sum_kr_kr_{k-m}-4T\sigma^2/N$, so \eqref{dh:eq} says that the two matrices are equal. This is \eqref{sx:regularityeq} with $d=\sqrt T\,\one$. In particular $H\one=4T(\sigma^2-\sigma^2)\one=0$. The vectors $(\cos(2\pi jk/N))_j$ and $(\sin(2\pi jk/N))_j$ are eigenvectors of $H$ with eigenvalue $h_k$, so by \eqref{dh:pos} $H$ is positive semidefinite of rank $N-1$, and its kernel is spanned by $\one$. Theorem~\ref{sx:regularity} gives the simplex.

The elements of $D_N$ preserve cyclic distances, hence preserve $K$, and by Proposition~\ref{sx:reflectiontest} they are induced by isometries. If the $c_m$ are distinct, a permutation that preserves $K$ preserves cyclic distances, in particular the pairs at distance one, so it is an automorphism of the $N$-cycle and lies in $D_N$. Finally, no element of $D_N$ acts on the vertices as a transposition when $N\ge5$: a nontrivial rotation fixes no vertex, and a reflection $i\mapsto s-i$ fixes at most two vertices and so interchanges at least two pairs. By Proposition~\ref{sx:reflectiontest}, $\Delta$ has no reflection symmetry.
\end{proof}

\begin{remark}
Let $\kappa_k=\sum_jr_j\cos(2\pi jk/N)$ be the eigenvalues of $K$. Since $H=4TK(I-\frac1N\one\one^T)K$, we have $h_k=4T\kappa_k^2$ for $k\ge1$. Under \eqref{dh:eq}, condition \eqref{dh:pos} is therefore equivalent to invertibility of $K$, since its constant eigenvalue $\sigma$ is positive. For $N=5$ the system is the family of Theorem~\ref{ld:d5}, with $c_1=p$ and $c_2=r$, and only the regular simplex occurs. We make no existence or nonexistence assertion here for $N=6,7,8$. The solution for $N=9$ in Table~\ref{dh:table} starts the certified family below.
\end{remark}

\begin{theorem}\label{dh:main}
For every $N$ with $9\le N\le201$, the system \eqref{dh:eq} has a solution with
\[
1<c_1<c_2<\dots<c_M<2,\; T>0,
\]
which satisfies \eqref{dh:pos}. Consequently, for every $n$ with $8\le n\le200$ there is an $n$-simplex whose Morley simplex is regular, whose isometry group is the dihedral group of order $2(n+1)$, and which has no hyperplane reflection symmetry.
\end{theorem}

\begin{proof}
The second statement follows from the first by Lemma~\ref{dh:lemma}. The first is checked by the script \texttt{check\_circulant.py}, separately for each $N$, by the method used for Theorem~\ref{hd:counterexample}. For $N=9$ the script starts from the rounded values in Table~\ref{dh:table}; for $N>9$ it starts from the solution for $N-1$, interpolated linearly at the points $m/N$, with the value $c_0=1$ at $0$. Newton's method in floating point, followed by a few simplified Newton steps in $90$-digit decimal arithmetic, gives an approximate solution, which is rounded to a point $x_0\in\mathbb Q^{M+1}$ with $60$ decimals. With a floating point approximate inverse $C$ of the Jacobian at $x_0$ and $\rho=10^{-30}$, the script verifies on $\mathcal I=x_0+[-\rho,\rho]^{M+1}$ that
\begin{equation}\label{dh:bounds}
\kappa<5\cdot10^{-11},\;\|CE(x_0)\|_\infty<10^{-60},
\end{equation}
with $\kappa$ defined as in \eqref{hd:bounds}. As in the proof of Theorem~\ref{hd:counterexample}, this gives a unique zero of $E$ in $\mathcal I$. On $\mathcal I$ the script then checks $1<c_1$, $c_M<2$, $T>0$, that the interval of $c_m$ lies below the interval of $c_{m+1}$ for every $m<M$, and that $h_k>0$ for $k=1,\dots,M$. For the last check, $\pi$ is enclosed using Machin's formula and the cosines by Taylor polynomials with an explicit bound for the remainder. All comparisons are exact.
\end{proof}

\begin{corollary}\label{dh:cor}
Conjecture~\ref{hd:all-dimensions} holds for $6\le n\le200$.
\end{corollary}

\begin{proof}
Combine Theorems~\ref{hd:counterexample}, \ref{hd:seven-counterexample} and~\ref{dh:main}.
\end{proof}

\begin{table}[ht]
\centering\small
\begin{tabular}{rlllll}\toprule
$N$&$c_1$&$c_M$&$T$&$\min_m(c_{m+1}-c_m)$&$\min_kh_k$\\\midrule
$9$&$1.5535062473$&$1.8993168010$&$0.3841947084$&$5.39\cdot10^{-2}$&$3.10\cdot10^{-1}$\\
$10$&$1.483843$&$1.925721$&$0.338986$&$2.52\cdot10^{-2}$&$1.78\cdot10^{-1}$\\
$20$&$1.236747$&$1.956117$&$0.164637$&$7.66\cdot10^{-3}$&$1.34\cdot10^{-2}$\\
$50$&$1.095864$&$1.960271$&$0.065478$&$1.25\cdot10^{-3}$&$6.98\cdot10^{-4}$\\
$100$&$1.048246$&$1.960701$&$0.032716$&$3.14\cdot10^{-4}$&$8.20\cdot10^{-5}$\\
$150$&$1.032242$&$1.960773$&$0.021808$&$1.40\cdot10^{-4}$&$2.38\cdot10^{-5}$\\
$200$&$1.024212$&$1.960797$&$0.016356$&$7.86\cdot10^{-5}$&$9.95\cdot10^{-6}$\\
$201$&$1.024092$&$1.960778$&$0.016274$&$1.56\cdot10^{-4}$&$9.80\cdot10^{-6}$\\\bottomrule
\end{tabular}
\caption{Rounded values for some members of the certified dihedral family. Every case satisfies the common bounds \eqref{dh:bounds}. For $N=9$ the other values are $c_2\approx1.7336022961$ and $c_3\approx1.8454039853$.}\label{dh:table}
\end{table}

For $N=9$ the four dihedral angles are approximately $117.107^\circ$, $89.733^\circ$, $68.027^\circ$ and $54.772^\circ$ for cyclic distances $1,2,3,4$. The computed solutions for $9\le N\le201$ suggest that the points $(m/N,c_m)$ approach an increasing profile. At the largest values of $N$, the quantities $N(c_1-1)$ and $NT$ are about $4.84$ and $3.27$, while $c_M$ is close to $1.961$. The minimum gap appears to have order $N^{-2}$, with a parity effect: $N^2\min_m(c_{m+1}-c_m)$ is about $3.14$ for $N=200$ and $6.29$ for $N=201$. The quantity $N^3\min_kh_k$ is close to $80$ in these cases. These are observations on the certified finite range, not asymptotic statements.

If a limiting profile exists, a natural normalization is $c_m\approx\gamma(m/N)$, where $\gamma$ is continuous, even and periodic with period one, and $\gamma(0)=1$. If also $NT\to\tau$, replacing the sums in \eqref{dh:eq} by integrals leads formally to
\[
\begin{gathered}
4\tau\Bigl(\int_0^1\gamma(y)\gamma(y-x)\,dy-\Bigl(\int_0^1\gamma(y)\,dy\Bigr)^2\Bigr)=\ph(\gamma(x)),\\
0\le x<1.
\end{gathered}
\]
At $x=0$ this says that $4\tau$ times the variance of $\gamma$ equals one. Existence of such a profile and convergence of the discrete solutions remain open. The computations suggest the following conjecture.

\begin{conjecture}\label{dh:conj}
For every $N\ge9$ the system \eqref{dh:eq} has a solution with $1<c_1<c_2<\dots<c_M<2$ and $T>0$ which satisfies \eqref{dh:pos}.
\end{conjecture}

By Lemma~\ref{dh:lemma} and Theorems~\ref{hd:counterexample} and~\ref{hd:seven-counterexample}, Conjecture~\ref{dh:conj} implies Conjecture~\ref{hd:all-dimensions}.

\section{The Johnson family}\label{jn:section}

This family uses $N=\binom k2$ vertices, so the dimension is $n=\binom k2-1$; for $k\ge9$ these are the dimensions $35,44,54,65,77,\dots$ Let $k\ge5$ and index the vertices by the two element subsets of $\{1,\dots,k\}$. We look for a matrix $K$ whose off-diagonal entry is $p$ if the two labels intersect and $r$ if they are disjoint; this is the pattern of the Johnson scheme $J(k,2)$.

\begin{lemma}\label{jn:reduction}
Such data give a simplex with regular Morley simplex provided $1<p,r<2$ and
\begin{align}
 h_0&=1+2(k-2)\ph(p)+\binom{k-2}{2}\ph(r)=0,\label{jn:h0}\\
 h_1&=1+(k-4)\ph(p)-(k-3)\ph(r)>0,\notag\\
 h_2&=1-2\ph(p)+\ph(r)>0,\notag\\
 h_1\kappa_2^2-h_2\kappa_1^2&=0,\; \kappa_1\kappa_2\ne0,\label{jn:equal}
\end{align}
where
\[
 \kappa_1=1+(k-4)p-(k-3)r,\; \kappa_2=1-2p+r.
\]
If $p\ne r$, the simplex has no hyperplane reflection symmetry.
\end{lemma}

\begin{proof}
Let $A$ be the adjacency matrix of intersecting pairs and $T$ the $k\times N$ incidence matrix of points and pairs. Then $T^TT=2I+A$ and $TT^T=(k-2)I+\one\one^T$, so $A$ has the eigenvalues $2(k-2)$, $k-4$, $-2$, with multiplicities $1$, $k-1$, $k(k-3)/2$. The matrix of disjoint pairs is $J-I-A$, where $J=\one\one^T$. Hence the eigenvalues of $H$ are $h_0,h_1,h_2$, and those of $K$ on the nonconstant eigenspaces are $\kappa_1,\kappa_2$. The conditions give $\ker H=\R\one$ and a common positive value $h_1/\kappa_1^2=h_2/\kappa_2^2=4\tau^2$, so \eqref{sx:regularityeq} holds with $d=\tau\one$, and Theorem~\ref{sx:regularity} gives the simplex.

For two distinct labels $X,Y$ there is a third label $Z$ that meets one of them but not the other. If $X,Y$ share an element, take one of the unshared elements together with an element outside $X\cup Y$; if they are disjoint, take an element of $X$ and an element outside $X\cup Y$. So when $p\ne r$ no transposition of two labels preserves $K$, and Proposition~\ref{sx:reflectiontest} excludes every reflection.
\end{proof}

For $p\ne r$ the full symmetry group is $S_k$: the maximal cliques of size $k-1>3$ in the intersection graph are the stars of pairs through a fixed element, so every automorphism permutes the $k$ stars and is induced by an element of $S_k$.

\begin{theorem}\label{jn:existence}
For every $k\ge9$ there is a simplex of dimension $\binom k2-1$ whose Morley simplex is regular and which has no hyperplane reflection symmetry. Consequently $(G_n)$ is false for $n=35,44,54,65,77,\dots$
\end{theorem}

\begin{proof}
Put $\epsilon=1/k$, $p_*=2\cos(2\pi/9)$, so that $p_*^3-3p_*+1=0$, and $w=\sqrt3$. Set $p=p_*+s\epsilon$. Equation \eqref{jn:h0} can be written as
\begin{equation}\label{jn:psi}
 \ph(r)=\epsilon\Psi,\;
 \Psi=-\frac{2\{\epsilon+2(1-2\epsilon)\ph(p)\}}{(1-2\epsilon)(1-3\epsilon)}.
\end{equation}
Write $r=w+\epsilon\Delta$. Then $\Delta$ satisfies
\begin{equation}\label{jn:delta}
 -3\Delta-\frac{3w}{2}\epsilon\Delta^2-\frac12\epsilon^2\Delta^3=\Psi.
\end{equation}
The following expressions extend continuously to $\epsilon=0$:
\begin{align*}
 A&=\epsilon h_1=\epsilon+(1-4\epsilon)\ph(p)-(1-3\epsilon)\epsilon\Psi,\\
 B&=h_2/\epsilon=(3p_*^2-3)s+3p_*s^2\epsilon+s^3\epsilon^2+\Psi,\\
 C_1&=\epsilon\kappa_1=p-r+\epsilon(1-4p+3r),\; C_2=\kappa_2=1-2p+r,\\
 \widetilde G&=AC_2^2-BC_1^2.
\end{align*}
For $\epsilon>0$, equation \eqref{jn:equal} is equivalent to $\widetilde G=0$. On the rectangle $7/10\le s\le23/20$, $0\le\epsilon\le1/16$, rational interval arithmetic gives
\begin{equation}\label{jn:signs}
 1<p<r<2,\; A,B>0,\; C_1,C_2<0,\;
 \widetilde G(7/10,\epsilon)>0>\widetilde G(23/20,\epsilon);
\end{equation}
the inequalities for $p$ and $r$ also hold at $\epsilon=0$. Here $p_*$ is enclosed in $(3/2,8/5)$ and $w$ in $(17/10,18/10)$ by $85$ bisection steps. The positive root $\Delta$ of \eqref{jn:delta} is enclosed by bisection in $[0,2]$, using that it increases with $-\Psi$ and decreases with $w$ and $\epsilon$. Adaptive subdivision gives the following result:
\begin{center}\small
\begin{tabular}{lrr}\toprule
Inequalities&Rectangles&Lower bound\\\midrule
$p-1,\ 2-r,\ r-p,\ A,\ B,\ -C_1,\ -C_2>0$&3&$.0415$\\
$\widetilde G(7/10,\epsilon)>0$&8&$.00935$\\
$-\widetilde G(23/20,\epsilon)>0$&12&$.000452$\\\bottomrule
\end{tabular}
\end{center}
For each $k\ge16$, the intermediate value theorem in $s$ gives a root of \eqref{jn:equal} with the properties required by Lemma~\ref{jn:reduction}.

The seven values $9\le k\le15$ are checked directly. For each interval of $p$ in the table below, let $r\in(1,2)$ be the solution of \eqref{jn:h0}. Interval arithmetic gives $h_1,h_2>0$, $\kappa_1,\kappa_2<0$ and $1<p<r<2$ on the whole interval, and $h_1\kappa_2^2-h_2\kappa_1^2$ is positive at the left end and negative at the right end.
\begin{center}\small
\begin{tabular}{cccccccc}\toprule
$k$&$9$&$10$&$11$&$12$&$13$&$14$&$15$\\\midrule
left end of $p$&$1.6676$&$1.6456$&$1.6312$&$1.6200$&$1.6112$&$1.6044$&$1.5984$\\
right end of $p$&$1.6680$&$1.6460$&$1.6316$&$1.6204$&$1.6116$&$1.6048$&$1.5988$\\\bottomrule
\end{tabular}
\end{center}
The intermediate value theorem and Lemma~\ref{jn:reduction} complete the proof.
\end{proof}

\begin{proposition}[An asymptotic branch]\label{jn:asymptotic}
For all sufficiently large $k$ there is a Johnson solution with
\[
 p=p_*+\frac{c}{k}+O(k^{-2}),\; r=\sqrt3+\frac{2}{3k}+O(k^{-2}),
\]
where
\[
 c=\frac{2+(1-2p_*+\sqrt3)^2/[2(p_*-\sqrt3)^2]}{3p_*^2-3}\approx0.836089.
\]
For this branch the dihedral angles between facets with intersecting labels tend to $120^\circ$, those between facets with disjoint labels tend to $90^\circ$, and
\[
 kh_2\longrightarrow\frac{(1-2p_*+\sqrt3)^2}{2(p_*-\sqrt3)^2}\approx1.379.
\]
\end{proposition}

\begin{proof}
At $\epsilon=0$, \eqref{jn:psi} and \eqref{jn:delta} give $\Psi=-2$ and $\Delta=2/3$, and
\[
 \widetilde G(s,0)=\tfrac12(1-2p_*+\sqrt3)^2-\bigl((3p_*^2-3)s-2\bigr)(p_*-\sqrt3)^2.
\]
Its only zero is $s=c$, where its derivative in $s$ is negative. The implicit function theorem, applied first to \eqref{jn:delta} and then to $\widetilde G=0$, gives an analytic branch $s=c+O(\epsilon)$, and the limits follow by substitution.
\end{proof}

For $k=9$, the interval used in the proof contains a solution with $p\approx1.6679131531$ and $r\approx1.7858080019$. The theorem proves existence, not the number of Johnson solutions for each $k$. For $9\le k\le20$, the dimension $\binom k2-1$ is at most $189$ and is also covered by Theorem~\ref{dh:main}, with a different symmetry group. For $k\ge21$, the Johnson construction provides an infinite set of dimensions beyond the finite dihedral range.

\section{Towards the remaining dimensions}\label{hd:directions}

By Corollary~\ref{dh:cor} and Theorem~\ref{jn:existence}, Conjecture~\ref{hd:all-dimensions} is open only for $n\ge201$ with $n\ne\binom k2-1$. We see three ways to proceed.

\subsection{The dihedral family for all \texorpdfstring{$N$}{N}}

A proof of Conjecture~\ref{dh:conj} would give examples in every remaining dimension. Extending the computation in Theorem~\ref{dh:main} reaches further finite ranges, but cannot prove the assertion for all $N$. One possible approach is to solve the integral equation in Section~\ref{dh:section}, prove invertibility of its linearization, and compare its solution quantitatively with the discrete equations \eqref{dh:eq}. The estimates would need to be uniform in $N$ and strong enough to meet the certified range. The Johnson family offers a simpler model: its reduced equation has a nondegenerate limit as $k\to\infty$. For the dihedral family, the behaviour of the proposed profile near $0$ also needs to be understood.

\subsection{Families with a simultaneous interchange of pairs}

The examples in dimensions six and seven are invariant under an involution that interchanges several pairs of vertices at once, with one fixed vertex when the number of vertices is odd. Such an involution reduces the number of equations without forcing any single transposition, and the reflection test of Proposition~\ref{sx:reflectiontest} can exclude transpositions by strict inequalities. One would like to find a branch of such solutions whose size grows with the number of pairs and to control it uniformly. We have not found such a branch.

\subsection{Other families}

Other permutation actions may reduce the regularity equations to a small number of spectral equations, as in the circulant and Johnson constructions. A useful family would have solutions for an unbounded set of dimensions, together with a positive kernel and a way to exclude every transposition. Finding such a family in dimensions not covered by the Johnson construction would give another route to Conjecture~\ref{hd:all-dimensions}.

\section{Open problems}\label{hd:open}

\begin{enumerate}[label=\arabic*.]
\item Prove Conjecture~\ref{ld:two-reflections}, that is, $(G_4)$ and $(G_5)$. In dimension four, are the regular simplex, $U_1$ and $U_2$ the only simplices with regular Morley simplex, up to similarity? In dimension three the corresponding question is equivalent to $(G_3)$ \cite{TranTetra}.
\item Prove Conjecture~\ref{dh:conj}, or at least Conjecture~\ref{hd:all-dimensions}. The smallest open dimension is $201$.
\item Does the dihedral system admit other nonregular branches? Can one construct a seven-simplex with regular Morley simplex whose only nonidentity isometry fixes two vertices and interchanges three pairs?
\item Is there a simplex with regular Morley simplex and trivial isometry group, or one with all edge lengths distinct? All nonregular examples certified in this paper have a nontrivial symmetry.
\item For fixed $n\ge3$, are there only finitely many similarity classes of $n$-simplices with regular Morley simplex?
\item How large is the neighbourhood in Theorem~\ref{lin:main}, and how does it depend on $n$?
\item Which of these simplices can be expressed by radicals? In dimension three $T_2$ can and $T_1$ cannot; in dimension four neither $U_1$ nor $U_2$ can. The dihedral and Johnson simplices are defined by small polynomial systems for each $N$ or $k$, and their Galois groups could be computed in small cases.
\item Extend the sector construction of \cite[Theorem 6]{TranActa} to higher dimensions and study the same questions for it.
\end{enumerate}

\begin{appendices}

\section{The polynomials for \texorpdfstring{$U_2$}{U2}}\label{app:U2}

In Theorem~\ref{ld:U2} the numbers $q$, $W=w^2$, $E=\varepsilon^2$ and $\rho=\delta\varepsilon/w$ are the values at $t=p$ of the following polynomials of degree seven:
\[
q(t)=\frac1{979353401265}\sum_{i=0}^7a_it^i,\;
W(t)=\frac1{1149701152527632835}\sum_{i=0}^7b_it^i,
\]
\[
E(t)=\frac1{165556965963979128240}\sum_{i=0}^7e_it^i,\;
\rho(t)=\frac1{16555696596397912824000}\sum_{i=0}^7r_it^i,
\]
with the coefficients
\begin{center}\footnotesize
\begin{tabular}{rrrrr}\toprule
$i$&$a_i$&$b_i$&$e_i$&$r_i$\\\midrule
0&$2044678924535$&$27872689682431037$&$-7235108445085873393$&$-3789659120408919888617$\\
1&$692731432034$&$-3156726082013655871$&$-44951627927193795556$&$-1693925628457883637113$\\
2&$-1726750815621$&$2240075227696022046$&$-19190583740504464020$&$4010686909907132162628$\\
3&$57418988298$&$1616487564855589476$&$38511387194611381104$&$1307911083224799214518$\\
4&$579692004792$&$-410185382810556864$&$18103956178091322036$&$-806372707562910773364$\\
5&$-61069470858$&$-72635960487906234$&$-4960215208001841096$&$-14857873806497543124$\\
6&$-20596391096$&$-4194250138476272$&$-1143622367558162540$&$12509643828607964960$\\
7&$-2086645424$&$326576872153132$&$-97258609107918560$&$1894695383395472012$\\\bottomrule
\end{tabular}
\end{center}
The polynomial $q(t)$ agrees at $t=p$ with the formula $q=\frac12\bigl(p+4-\sqrt{3(4-p^2)}\bigr)$ of Theorem~\ref{ld:U2}. The polynomials $A$ and $C$ for $U_1$ have coefficients with about $75$ digits and are given in the file \texttt{check\_U1\_U2.py}.

\section{Verification of the computations}\label{sec:reproducibility}

The computer assisted steps are checked by the following Python scripts, available as ancillary files with the arXiv version. They use only the standard library. All inequalities used in proofs are decided in exact rational arithmetic, or with intervals whose endpoints are rational numbers rounded outwards; floating point numbers serve only to find approximate solutions and approximate inverses. The scripts should not be run with Python's \texttt{-O} option, which disables assertions.

From the \texttt{anc} directory, the command \texttt{python verify\_all.py} runs all six checks and stops if any fails. The shared toolkit rejects runs with assertions disabled. The accompanying \texttt{README.md} gives the individual commands and explains how to extend the finite circulant range. The programs have been checked with Python~3.12; no additional packages are required.

\begin{center}\small
\begin{tabular}{p{.3\textwidth}p{.62\textwidth}}\toprule
File&Content\\\midrule
\texttt{morley\_tools.py}&polynomial arithmetic, Sturm sequences, rational intervals, factorization modulo a prime, intervals with outward rounding, enclosures of $\pi$ and of cosines, the contraction test of \eqref{hd:bounds}\\
\texttt{check\_linearization.py}&a numerical illustration of Proposition~\ref{sx:gram} and Theorem~\ref{lin:main} (not needed for the proofs)\\
\texttt{check\_cyclic5.py}&the two sign conditions in the proof of Theorem~\ref{ld:d5}\\
\texttt{check\_U1\_U2.py}&the exact identities, Sturm counts and interval bounds of Theorems~\ref{ld:U2} and~\ref{ld:U1}, and the factorizations of Theorem~\ref{ld:galois}\\
\texttt{check\_dim6\_dim7.py}&the bounds in the proofs of Theorems~\ref{hd:counterexample} and~\ref{hd:seven-counterexample}\\
\texttt{check\_circulant.py}&Theorem~\ref{dh:main} for $9\le N\le201$; a larger bound can be passed as an argument\\
\texttt{check\_johnson.py}&the inequalities \eqref{jn:signs} and the cases $9\le k\le15$ of Theorem~\ref{jn:existence}\\\bottomrule
\end{tabular}
\end{center}

\end{appendices}

\section*{Acknowledgements}
\phantomsection
\addcontentsline{toc}{section}{Acknowledgements}

The author thanks Tran Manh Dung and Nguyen Xuan Tho for their help with the computations. The AI assistants ChatGPT and Claude were used for exploratory computations, for checking algebra and in preparing the manuscript; the author takes full responsibility for the content.

\phantomsection

\begin{flushright}
Quang Hung Tran,\\
High School for Gifted Students,\\
Vietnam National University, Hanoi, Vietnam,\\
Email: \href{mailto:tranquanghung@hus.edu.vn}{tranquanghung@hus.edu.vn}\\
ORCID: \href{https://orcid.org/0000-0003-2468-4972}{0000-0003-2468-4972}
\end{flushright}


\addcontentsline{toc}{section}{References}
\begin{thebibliography}{98}

\bibitem{Braude}
E.~J. Braude, \emph{Generalizing the Morley trisector and various theorems with realizability computations}, arXiv:1603.03463, 2016.

\bibitem{Child}
J.~M. Child, \emph{Proof of Morley's theorem}, Math. Gaz. \textbf{11} (1923), 171.

\bibitem{CoxeterGreitzer}
H.~S.~M. Coxeter and S.~L. Greitzer, \emph{Geometry Revisited}, New Mathematical Library, Vol.~19, Mathematical Association of America, Washington, DC, 1967.

\bibitem{DergiadesTran}
N.~Dergiades and Q.~H. Tran, \emph{On some extensions of Morley's trisector theorem}, J. Geom. Graph. \textbf{24} (2020), 197--205.

\bibitem{GrinbergOrhon}
E.~L. Grinberg and M.~Orhon, \emph{Morley trisectors and the law of sines with reflections}, Amer. Math. Monthly \textbf{128} (2021), 163--167.

\bibitem{Hashimoto}
Y.~Hashimoto, \emph{A short proof of Morley's theorem}, Elem. Math. \textbf{62} (2007), 121.

\bibitem{Krawczyk}
R.~Krawczyk, \emph{Newton-Algorithmen zur Bestimmung von Nullstellen mit Fehlerschranken}, Computing \textbf{4} (1969), 187--201.

\bibitem{Morley}
F.~Morley, \emph{On the metric geometry of the plane $n$-line}, Trans. Amer. Math. Soc. \textbf{1} (1900), 97--115.

\bibitem{Neumaier}
A.~Neumaier, \emph{Interval Methods for Systems of Equations}, Encyclopedia of Mathematics and its Applications, Vol.~37, Cambridge University Press, Cambridge, 1990.

\bibitem{OakleyBaker}
C.~O. Oakley and J.~C. Baker, \emph{The Morley trisector theorem}, Amer. Math. Monthly \textbf{85} (1978), 737--745.

\bibitem{Rump}
S.~M. Rump, \emph{Verification methods: rigorous results using floating point arithmetic}, Acta Numer. \textbf{19} (2010), 287--449.

\bibitem{SvrtanVeljan}
D.~Svrtan and D.~Veljan, \emph{Side lengths of Morley triangles and tetrahedra}, Forum Geom. \textbf{17} (2017), 123--142.

\bibitem{TranActa}
Q.~H. Tran, \emph{Morley's trisector theorem for isosceles tetrahedron}, Acta Math. Hungar. \textbf{165} (2021), 308--315. \href{https://doi.org/10.1007/s10474-021-01184-0}{doi:10.1007/s10474-021-01184-0}.

\bibitem{TranSynthetic}
Q.~H. Tran, \emph{A synthetic proof of the Morley trisector theorem using congruent and similar triangles}, Elem. Math. \textbf{80} (2025), 70--72.

\bibitem{TranTetra}
Q.~H. Tran, \emph{The converse problem for the Morley tetrahedron: counterexamples, conjectures, and partial results}, preprint, 2026, companion paper.

\end{thebibliography}
\end{document}